\documentclass[11pt]{amsart}

\usepackage{ae,aecompl}
\usepackage{esint}
\usepackage{graphicx}
\usepackage[all,cmtip]{xy}
\usepackage{amsmath, amscd}
\usepackage{booktabs}
\usepackage{amsthm}
\usepackage{amssymb}
\usepackage{amsfonts}
\usepackage{qsymbols}
\usepackage{latexsym}
\usepackage{mathrsfs}
\usepackage{cite}
\usepackage{color}
\usepackage{url}
\usepackage{enumerate}
\usepackage[english]{babel}

\usepackage{verbatim}
\usepackage[draft=false, colorlinks=true]{hyperref}

\allowdisplaybreaks

\usepackage{ae,aecompl}
\usepackage{esint}
\usepackage{graphicx}
\usepackage[all,cmtip]{xy}
\usepackage{amsmath, amscd}
\usepackage{booktabs}
\usepackage{amsthm}
\usepackage{amssymb}
\usepackage{amsfonts}
\usepackage{qsymbols}
\usepackage{latexsym}
\usepackage{mathrsfs}
\usepackage{cite}
\usepackage{color}
\usepackage{url}
\usepackage{enumerate}
\usepackage[english]{babel}

\usepackage{verbatim}
\usepackage[draft=false, colorlinks=true]{hyperref}

\allowdisplaybreaks

\usepackage{hyperref}
\usepackage{url}
\usepackage{setspace}
\usepackage{float}

\allowdisplaybreaks
\newcommand{\R}{{\mathbb R}}       % Field of real numbers

\newcommand{\Z}{{\mathbb Z}}       % Ring of integer numbers

\newtheorem{theorem}{Theorem}[section]
\newtheorem{lemma}[theorem]{Lemma}
\newtheorem{remark}[theorem]{Remark}

\newtheorem{proposition}[theorem]{Proposition}

\newtheorem*{theorem*}{Theorem}

\theoremstyle{definition}

\numberwithin{equation}{section}

\makeatletter
\DeclareRobustCommand\widecheck[1]{{\mathpalette\@widecheck{#1}}}
\def\@widecheck#1#2{%
    \setbox\z@\hbox{\m@th$#1#2$}%
    \setbox\tw@\hbox{\m@th$#1%
       \widehat{%
          \vrule\@width\z@\@height\ht\z@
          \vrule\@height\z@\@width\wd\z@}$}%
    \dp\tw@-\ht\z@
    \@tempdima\ht\z@ \advance\@tempdima2\ht\tw@ \divide\@tempdima\thr@@
    \setbox\tw@\hbox{%
       \raise\@tempdima\hbox{\scalebox{1}[-1]{\lower\@tempdima\box
\tw@}}}%
    {\ooalign{\box\tw@ \cr \box\z@}}}
\makeatother

\begin{document}
	
	\title{Helical maximal function and weighted estimates}

\author[Ghosh]{Abhishek Ghosh}

\address[Abhishek Ghosh]{Department of Mathematics, Indian Institute of Technology Madras, Chennai, 600036, India.}

\email{abhi@iitm.ac.in, abhi170791@gmail.com}

\author[Shuin]{Kalachand Shuin}

\address[Kalachand Shuin]{Department of Mathematics, Indian Institute of Science, Bengaluru-560012, India.}
\email{kalachands@iisc.ac.in, shuin.k.c@gmail.com}
\thanks{AG gratefully acknowledges the 
support by the Industrial Consultancy and Sponsored Research (IC\& SR), Indian Institute of Technology Madras for the
New Faculty Initiation Grant RF25261459MANFIG009296. KS gratefully acknowledges the support of Inspire Faculty Award of Department of Science
and Technology, Govt. of India, Registration no. IFA23-MA191.
}

	\begin{abstract}
In this article, we characterize the range of $\alpha$ for which the helical maximal function is bounded from $L^p(|x|^\alpha)$ to itself for $3<p<\infty$. Our result is optimal for $4\leq p<\infty,$ except possibly at end-points. 
	\end{abstract}

	\maketitle
	
\section{introduction}
Let $\gamma:I\to \R^3$ be a smooth non-degenerate curve, i.e. 
\begin{eqnarray}\label{nondegenerate}
    |\det(\gamma'(s),\gamma''(s),\gamma'''(s))|\geq c_0,
\end{eqnarray}
for some positive constant $c_0$.  This non-degeneracy condition is equivalent to the non-vanishing curvature and non-zero torsion of the space curve $\gamma.$ We consider the following averaging operator 
\[A_tf(x):=\int_{I} f(x-t\gamma(s)) \chi(s)~ds,\]
where $\chi$ is a bump function supported in the interior of $I$. Some well known examples are the helix $\gamma(s)=(\cos s, \sin s, s )$ on $\mathbb{R}^3,$ the moment curve $\gamma (s)=(s,s^2/2,s^3/6)$ etc.
Our main objects of interest are the associated maximal functions, namely
$$Mf(x):=\sup_{t>0}|A_{t}f(x)|,$$
the full maximal function and the lacunary maximal function $M_{\text{lac}}f(x):=\sup_{k\in\mathbb{Z}}|A_{2^{k}}f(x)|.$ Also, the local maximal operator is defined as $M_{\text{loc}}f(x)=\sup_{1\leq t\leq2}|A_{t}f(x)|.$ In recent times, the study of averaging operators and maximal functions over non-degenerate curves has gained a lot of attraction; we start with the following breakthrough result proved in \cite{BeltranAJM} and \cite{KoLeeOh}.

\medskip

\noindent \textbf{Theorem A.}\cite{BeltranAJM, KoLeeOh}
\emph{Let $\gamma:I\to \mathbb{R}^3$ be a smooth non-degenerate curve. Then the maximal function $M$ is bounded from $L^p(\mathbb{R}^3)$ to itself if and only if $p>3.$}

A previous breakthrough in this direction was by Pramanik and Seeger in \cite{PramanikAJM} where the authors proved Theorem~A for large values of $p.$  In order to provide the right context to the above result, in an influential work in \cite{SteinSph}, Stein proved that the spherical maximal operator $\mathcal{S},$ defined  by
\[
\mathcal{S}f(x)=\sup_{t>0}\biggl|\int_{{\mathbb{S}}^{n-1}}f(x-ty)\, d \sigma_{n-1}(y)\biggr|, \quad x \in \mathbb R^n,
\]
where  $d \sigma_{n-1}$ is the rotationally invariant measure on $\mathbb{S}^{n-1},$ is bounded on $L^p(\R^n),$ for $n\geq 3$ and $p>\frac{n}{n-1}$ and initiated the study of averaging operators over lower dimensional manifolds. He also showed examples showing that the result is sharp. Subsequently, in \cite{Bourgain} Bourgain settled the problem for dimension two, concluding that the circular maximal function is bounded on $L^p(\R^2)$ for $p>2,$ and maximal functions over smooth non-degenerate curves can be seen as higher dimensional analogue of the circular maximal function. Theorem A thus provides the complete counterpart of Bourgain's result on $\mathbb{R}^3.$ 

The purpose of this present work is to quantify the range of $\alpha$ for which the maximal functions over non-degenerate curves map $L^p(|x|^\alpha)$ to itself on $\mathbb{R}^3$ i.e. we address the following question:
 
 \medskip
 
 \noindent \textbf{Question A.} \emph{ Let $\gamma:I\to \mathbb{R}^3$ be as in Theorem A.  Suppose $\alpha\in \mathbb{R}$. Then what is the optimal range of $\alpha$ such that $M$ maps $L^p(|x|^\alpha)$ to itself?}

 In this paper, we provide the following answer to this question.
 
\subsection{Statement of main result}

We start with stating necessary conditions on $\alpha$ such that the averaging operators $A_t,$ and the maximal function $M$ map $L^p(|x|^\alpha)$ to itself. 

\begin{theorem}
\label{Thm:nece}
Let $\gamma: I\to \mathbb{R}^3$ be a smooth non-degenerate curve, i.e., $\gamma$ satisfies \eqref{nondegenerate}. 
\begin{enumerate}
    \item 
Let $$\int_{\R^3} |A_{t}f(x)|^p |x|^{\alpha}\leq C \int_{\R^3} |f(x)|^p |x|^{\alpha},$$
holds for all $f$ with constant independent of $f,$ then we must have $-1\leq \alpha\leq (p-1).$
\item 
Let the maximal operator $M$ maps $L^p(|x|^\alpha)$ to itself, then $\alpha$ necessarily satisfies
$$-1\leq \alpha\leq (p-3).$$
In fact, the same conclusion holds if we replace $M$ by local maximal function $M_{\text{loc}}.$
\end{enumerate}
\end{theorem}

Our second theorem concerns proving weighted estimates for the afore-mentioned range of $\alpha.$ 

\begin{proposition}\label{single}
    Let $\gamma:I\to \R^3$ be a smooth non-degenerate curve and $1<p<\infty$. Let $-1\leq \alpha\leq p-1,$ then $A_{t}$ is bounded from $L^p(|x|^{\alpha})$ itself.  
\end{proposition}

\begin{theorem}\label{full}
      Let $\gamma:I\to \R^3$ be a smooth non-degenerate curve. Then we have the following weighted estimates:

\begin{enumerate}
    \item
for $4\leq p<\infty$, $M$ is bounded in $L^p(|x|^{\alpha})$ if $-1< \alpha<p-3$. Moreover, the above result is sharp possibly except for $\alpha=-1$.
\item 
for $3<p<4,$ $M$ is bounded in $L^p(|x|^{\alpha})$ if $-\frac{5}{2}+\frac{6}{p}< \alpha<p-3$.
\end{enumerate}
\end{theorem}
We are able to prove the complete range of boundedness for $4\leq p<\infty$ but for the range $3<p<4$ our result is not complete. For the lacunary maximal operator $M_{\text{lac}}$ we obtain the complete range of $\alpha$ as stated in the next theorem.

\begin{theorem}\label{lacunary}
    Let $\gamma:I\to \R^3$ be a smooth non-degenerate curve. Then for $1<p<\infty$, $M_{\text{lac}}$ is bounded in $L^p(|x|^{\alpha})$ if and only if $-1\leq \alpha<p-1$. 
\end{theorem}

Now for a proper perspective on weighted estimates, we mention that in the paper \cite{Lacey} Lacey proved sparse bounds for spherical maximal operators for dimension $n\geq 2$ and the sparse bounds are well-known to imply a range of weighted estimates with weights $w$ belonging to the Muckenhoupt class $\mathcal{A}_p.$ However, it was mentioned in \cite{Lacey} that that the sparse bounds fail to recover the complete range of $\alpha$ for which the spherical maximal operator maps $L^p(|x|^\alpha)$ to itself. Therefore, the study of weighted boundedness of singular maximal operators with weights of the form $|x|^\alpha,$ known as power weights, warrants a different approach. Significantly, in the context of the the spherical maximal operator, in a well-known article by Duoandikoetxea and Vega in \cite{JavierVega}, the authors obtained a characterization of $\alpha$ for which the spherical maximal operator $\mathcal{S}$ is bounded from $L^p(|x|^\alpha)$ to itself. Their result states that the spherical maximal operator maps $L^p(|x|^\alpha)$ to itself provided $p>\frac{n}{n-1}$ and $-(n-1)<\alpha<(n-1)(p-1)-1$ and their result is sharp up to the left end-point. Very recently, in \cite{LeeGeom} the author proved the end-point result that $\mathcal{S}$ is bounded on $L ^p(|x|^{1-n})$ when $p\geq 2, n\geq 3,$ and $\mathcal{S}$ is bounded on $L^p(|x|^{-1})$ if and only if $p>2, n=2.$ The reader should note that in higher dimensions the result is still open for $p$ between $\frac{n}{n-1}$ and $2.$       

Now concerning the helical maximal function, in contrast with the spherical maximal operator, the decay of the Fourier multiplier $\mu_{t}(\xi):=\int e^{-i\langle t\gamma(s), \xi)\rangle}\, ds$ is very limited, precisely $|\mu_{t}(\xi)|\lesssim (1+t|\xi|)^{-1/3},$ therefore fixed time decay estimate are merely insufficient for estimating the maximal function, and this makes the analysis for helical maximal function much more complicated. Essentially, the first remarkable result in this direction was by Pramanik and Seeger in \cite{PramanikAJM}, where the authors proved non-trivial gain over the fixed time estimates using local-smoothing estimates. Set $\Sigma_{1}=\{\xi: \langle \xi, \gamma'(s)\rangle=0\,\text{for some}\, s\in I\}$ and $\Sigma_{2}=\{\xi: \langle \xi, \gamma'(s)\rangle=\langle \xi, \gamma''(s)\rangle=0, \text{for some}\, s\in I\}.$ The analysis in \cite{PramanikAJM} is in two parts, first localizing the frequency near the cone $\Sigma_{1}$ and away from $\Sigma_{2},$ and then localizing near $\Sigma_{2}.$ In the first case, since one obtains a better decay, the analysis is like that of the circular maximal function done by Mockenhoupt--Seeger--Sogge in \cite{MSS}. To handle the most difficult part of their analysis, the authors in \cite{PramanikAJM} analyzed the binormal cone $\Sigma_{2},$ by introducing an additional localization depending on the distance from $\Sigma_{2},$ and eventually combined these estimates by decoupling inequalities. This analysis was further sharpened in \cite{BeltranAJM} where the authors used appropriate square functions, Nikodym maximal estimates and proved the sharp result Theorem~A. An independent approach was taken by Ko--Lee--Oh in \cite{KoLeeOh} who also obtained Theorem A. The same authors first obtained non-trivial $L^p$ boundedness of $M$ in dimension $n\geq4$ \cite{Forummathpi}. Very recently, in \cite{Changkeon} the authors proved the sharp local smoothing conjecture for averaging operators with non-degenerate curves in all dimensions which extends the seminal work of Guth--Wang--Zhang in \cite{Guth2020} where the authors proved the local-smoothing conjecture in dimension two. Surprisingly, the resolution of the conjecture shows that even the sharp local smoothing estimates in \cite{Changkeon} are not sufficient to prove sharp $L^p$ estimates for the maximal operator $M$ for dimension $n\geq 5.$ The proof of our main Theorem~\ref{full} uses the local smoothing estimates as well as the $L^p-L^q$ smoothing estimates from \cite{BeltranPLMS}, however we point out a key subtlety by noting that we can construct functions $f$ for which $\|M(f)\|_{L^p(|x|^{-1})}\simeq \delta^{2/p}\simeq \|f\|_{L^p(|x|^{-1})}$ for very small $\delta>0,$ and this suggests a further spatial decomposition at the scale $\delta(=2^{j)}$ should be combined with the $L^4-L^6$ local smoothing estimate from \cite{BeltranPLMS} to obtain Theorem~\ref{full}~Part (1), and this is done in Lemma~\ref{deltagainfor4}. However, for $3<p<4,$ we use the $L^p$ local smoothing from \cite{Changkeon} but we can only go up to the weight $|x|^{-\frac{5}{2}+\frac{6}{p}}$ in this case. We conclude this section by highlighting the importance of our results, namely, Theorem 1.3, in comparison with the weighted estimates obtained by sparse operator methods from \cite{BeltranPLMS} where the authors obtained sharp(except end-points) $L^p$ improving properties for the local maximal operator and as consequence, one can obtain sparse operator bounds from \cite{BeltranPLMS}.
\subsection{Comparison with previously known results}
In \cite{BeltranPLMS}, the authors mentioned that using the $L^p$ improving properties of the local maximal function  one can deduce the following theorems.
Define $\mathcal{T}'=\{(1/r,1-1/s): (1/r,1/s)\in\mathcal{T}\}$, where $\mathcal{T}$ is the triangle with vertices $(0,0), (\frac{1}{3},\frac{1}{3})$ and $(\frac{1}{4},\frac{1}{6})$.
\begin{theorem}[\cite{BeltranPLMS}]\label{sparsedomination}
    Let  $f,g\in \mathcal{S}(\mathbb{R}^3)$ and $(1/r,1/s)\in \mathcal{T}'$. Then 
    \[|\langle Mf,g\rangle |\lesssim \sum_{Q\in\mathcal{Q}}|Q|\langle f\rangle_{r,Q}\langle g\rangle_{s,Q},\]
    where $\mathcal{Q}$ is a sparse family.
\end{theorem}
As an application of the above theorem, one can obtain the following weighted estimates.
\begin{theorem}\label{weightsfromsparse}
    Let $3<p<\infty.$ Let $\phi_{\mathrm{full}}$ be a real valued function such that $\frac{1}{\phi_{\mathrm{full}}}$ is a piecewise linear function on $[0, 1]$ whose graph connects the points $(0,1), (1/4,5/6), (1/3,2/3).$ Then we have 
    \[\Vert Mf\Vert_{L^p(\omega)}\leq C_{p, \omega} \Vert f\Vert_{L^p(\omega)}, \]
    for $\omega\in \mathcal{A}_{\frac{p}{r}}\cap RH_{\Big(\frac{\phi_{\mathrm{full}}(1/r)'}{p}\Big)'}$ and $3<r<p<\phi_{\mathrm{full}}(1/r)'.$
\end{theorem}
In the above $RH_{\rho}$ denote the reverse H\"{o}lder class of order $\rho>0$. It is also well known that $\omega\in \mathcal{A}_{p}\cap RH_{\rho}\iff \omega^{\rho}\in \mathcal{A}_{{\rho}(p-1)+1}$. Now we specialize Theorem~\ref{weightsfromsparse} for the power weights.
\begin{proposition}\label{comparison1}
    Let $4<p<\infty$. Then we have 
     \[\Vert Mf\Vert_{L^p(|x|^{\alpha})}\leq C_{\omega, p} \Vert f\Vert_{L^p(|x|^{\alpha})}, \]
     for $-1<\alpha<\max\{\frac{p-3}{2},\frac{3(p-4)}{4}\}$.
\end{proposition}
\begin{remark}\label{pgreaterthan6}
  At this point we remark that our Theorem \ref{full} implies that 
    \[\Vert Mf\Vert_{L^p(|x|^{\alpha})}\leq C_{p, \alpha} \Vert f\Vert_{L^p(|x|^{\alpha})}, \]
    for $-1<\alpha<p-3$, for $4\leq p<\infty$ which is a strict improvement over Proposition~\ref{comparison1}. 
\end{remark}

\begin{proof}[Proof of Proposition \ref{comparison1}]

Observe that
\begin{equation*}
\frac{1}{\phi_{ \mathrm{full} }(t)}=\begin{cases} 1-\frac{2t}{3}, & \textrm{ if  $ \, 0<t\leq \frac{1}{4},$}\\
\frac{4}{3}-2t, & \textrm{ if $ \, \frac{1}{4}\leq t<\frac{1}{3},$}
\end{cases}    
\end{equation*}
and its H\"older conjugate is
\begin{align*}
\frac{1}{(\phi_{ \mathrm{full} }(t))'}
=\begin{cases}
\frac{2t}{3}, & \textrm{ if  $ \, 0<t\leq \frac{1}{4},$}\\
2t-\frac{1}{3}, &   \textrm{ if $ \, \frac{1}{4}\leq t<\frac{1}{3}.$}    
\end{cases}    
\end{align*}
Let $4<p<\infty.$ Choose $r=p-\epsilon>4$ and the range of $\epsilon$ will be specified by the Theorem~\ref{weightsfromsparse}. To apply Theorem~\ref{weightsfromsparse} we must have $r<p<\phi_{\mathrm{full}}(\frac{1}{r})',$ therefore $p<\phi_{\mathrm{full}}(\frac{1}{p-\epsilon})'=\frac{3(p-\epsilon)}{2},$ which gives $0<\epsilon<\frac{p}{3}.$ Consequently, $0<\epsilon<\min\{p/3, p-4\}.$
\begin{comment}
 Since $p>4$, we choose $r<p$ with $(\frac{1}{r},\frac{1}{\phi_{full}(\frac{1}{r})})\in \Bar{\mathcal{T}'}$ such that $r<p<\phi'_{full}(\frac{1}{r})$. Therefore, we consider $r=p-\epsilon$ $(\epsilon>0)$ such that $p-\epsilon<\phi'_{full}(\frac{1}{p-\epsilon})$. Now we have two choices for $\phi_{full}(\frac{1}{r})$. When $(\frac{1}{r},\frac{1}{\phi_{full}(\frac{1}{r})})$ lies on the line segment joining $(0,1)$ and $(\frac{1}{4},\frac{5}{6})$ vertices, we get $0<\epsilon<p-4$ and
    \begin{eqnarray*}
        \frac{1}{\phi_{full}(1/r)}=1-\frac{2}{3r}
        \implies\phi'_{full}(\frac{1}{r})=\frac{3r}{2}=\frac{3(p-\epsilon)}{2}.
    \end{eqnarray*}
    Note that Theorem \ref{weightsfromsparse} requires $p<\phi'_{full}(\frac{1}{r})=\frac{3(p-\epsilon)}{2}$, i.e. $\epsilon<p/3$. So the choices for $\epsilon$ is .  
\end{comment}
Now by Theorem  \ref{weightsfromsparse} we have
    \begin{eqnarray*}
        &&|x|^{\alpha}\in \mathcal{A}_{\frac{p}{p-\epsilon}}\cap RH_{\Big(\frac{\phi_{\mathrm{full}}(\frac{1}{r})'}{p}\Big)'}.
        \end{eqnarray*}
        The above condition will imply the following range of $\alpha$,
        \begin{eqnarray*}
        -\frac{p-3\epsilon}{p-\epsilon}<\alpha<\frac{3\epsilon}{p-\epsilon}.
    \end{eqnarray*}
    Taking $\epsilon\to 0^{+}$ we get $-1<\alpha\leq0.$ For $p\geq 6,$ we can take $\epsilon\to (p/3)^{-}$ to obtain that $\alpha<3/2.$ For $4<p<6,$ we can take $\epsilon\to (p-4)^{-}$ thus yielding $\alpha<\frac{3(p-4)}{4}.$

 Again by choosing $3<r=p-\epsilon<4<p,$ we obtain the permissible range of $\epsilon$ to be $p-4\leq \epsilon<p-3.$ Moreover, the condition $p<\phi_{\mathrm{full}}(\frac{1}{r})'$ in this range implies that $p<(\frac{2}{r}-\frac{1}{3})^{-1}=\frac{3(p-\epsilon)}{6-(p-\epsilon)},$ that is, $\epsilon<\frac{p(p-3)}{p+3}.$ Now Theorem \ref{weightsfromsparse} implies that
     \begin{eqnarray*}
        |x|^{\alpha}\in \mathcal{A}_{\frac{p}{p-\epsilon}}\cap RH_{\Big(\frac{\phi_{\text{full}}(\frac{1}{r})'}{p}\Big)'}.
          \end{eqnarray*}
Consequently, $\alpha$ satisfies
          \begin{eqnarray*}
           -\frac{(p+3)(p-\epsilon)-6p}{p-\epsilon}<\alpha< \frac{3\epsilon}{p-\epsilon}\label{alpharange}.
     \end{eqnarray*}
     Letting $\epsilon\to (p-4)^{+}$ we get $-(3-\frac{p}{2})<\alpha<\frac{3(p-4)}{4},$ and by letting $\epsilon\to (\frac{p(p-3)}{p+3})^{-}$ we obtain the range $0\leq \alpha<\frac{p-3}{2}$. Hence, combining the above results, we get $-1<\alpha<\max\{\frac{p-3}{2},\frac{3(p-4)}{4}\}.$
\end{proof}
Now we focus on showing the comparison for the range $3<p<4.$ We record the following range of $\alpha$ that one can obtain from Theorem~\ref{weightsfromsparse}.

\begin{proposition}\label{comparison2}
     Let $3<p\leq4$. Then we have 
     \[\Vert Mf\Vert_{L^p(|x|^{\alpha})}\leq C_{\omega, p}\, \Vert f\Vert_{L^p(|x|^{\alpha})}, \]
     for $3-p<\alpha<\frac{p-3}{2}$.
\end{proposition}
As before, the Theorem \ref{full} implies that 
    \[\Vert Mf\Vert_{L^p(|x|^{\alpha})}\leq C_{\omega, p} \Vert f\Vert_{L^p(|x|^{\alpha})}, \]
    for $\frac{6}{p}-\frac{5}{2}<\alpha<p-3$, for $3<p\leq 4,$ which is a strict improvement over Proposition~\ref{comparison2}.
\begin{proof}[Proof of Proposition \ref{comparison2}]
   Let $3<p<4.$ Consider $r=p-\epsilon$, where   $\epsilon>0$ will be chosen to satisfy $r<p<\phi_{\mathrm{full}}(1/r)'$. Therefore, we have 
    \begin{eqnarray*}
        \frac{1}{(\phi_{\text{full}}(1/r))'}=\frac{6-(p-\epsilon)}{3(p-\epsilon)}.
    \end{eqnarray*}
    Further, $\epsilon$ has to satisfy the following inequality
    \begin{eqnarray*}
        &&p-\epsilon<p<\frac{3(p-\epsilon)}{6-(p-\epsilon)}\\
        \implies&& 6p-p(p-\epsilon)<3(p-\epsilon)\\
        \implies&& \epsilon<p-\frac{6p}{p+3}.
    \end{eqnarray*}
    Therefore, we have $0<\epsilon<\frac{p(p-3)}{p+3}$. 
    Theorem \ref{weightsfromsparse} implies that $M$ maps $L^p(|x|^{\alpha})$ to itself for $|x|^{\alpha}\in \mathcal{A}_{\frac{p}{p-\epsilon}}\cap RH_{(\frac{3(p-\epsilon)}{p(6-(p-\epsilon))})'}.$ Now, using the properties of Reverse H\"{o}lder class  $RH_{\rho}$ and $\mathcal{A}_p$ class we get the following range of $\alpha$ 
    \begin{eqnarray*}
        -\frac{(p+3)(p-\epsilon)-6p}{p-\epsilon}<\alpha<\frac{3\epsilon}{p-\epsilon}.
    \end{eqnarray*}
    Taking $\epsilon\to 0^{+}$ we get 
    \begin{eqnarray*}
        &&-\frac{(p+3)p-6p}{p-\epsilon}<\alpha\leq 0\\
        \implies&& -(p-3)<\alpha\leq0.
    \end{eqnarray*}
    On the other hand taking $\epsilon\to (\frac{p(p-3)}{p+3})^{-}$ we get 
    \begin{eqnarray*}
        &&0\leq \alpha<\frac{3p(p-3)/(p+3)}{p-\frac{p(p-3)}{p+3}}\\
        \implies&& 0\leq \alpha<\frac{p-3}{2}.
    \end{eqnarray*}
    Therefore, combining the above two estimates we get 
    $3-p<\alpha<\frac{p-3}{2}$. This completes the proof of Proposition~\eqref{comparison2}.
\end{proof}
We believe the above comparison shows the novelty of our results. In the Section~\ref{Necessary}, we prove our necessary conditions. In Section~\ref{proof:main} we prove Theorem~\ref{full}  and Theorem~\ref{lacunary}.         
\section{Necessary conditions}
\label{Necessary}
In this section, we prove several necessary conditions on $\alpha$ for the boundedness of the averaging operators $A_t$ and for the maximal operator $M$ on $L^p(|\cdot|^{\alpha}).$ 

\begin{proof}[Proof of Theorem~\ref{Thm:nece}, Part(1)] We will write the proof for $t=1,$ however it can be easily modified for other values of $t$ and the averaging operator $A_1$ will simply be written as $A.$ Let $\delta$ be a small positive number and consider the set $$N(\delta):=\{x\in\R^3: |x+\gamma(s)|<\delta\, \text{for some}\, s\in I\}.$$ 
Then observe that for any $x\in N(\delta),$ $|x|\geq |\gamma(s)|-|x-\gamma(s)|,$ concluding $|x|\gtrsim 1,$ and hence $\|\chi_{N_{\delta}}\|_{L^p(|x|^\alpha)}^p\simeq |N(\delta)|\simeq \delta^2.$ On the other hand, for all $x$ satisfying $|x|\leq \delta,$ $x-\gamma(s)\in N(\delta)$ for all $s\in I,$ therefore,
$$A(\chi_{N(\delta)})(x)=\int \chi_{N(\delta)}(x-\gamma(s))\chi(s)\, ds\gtrsim 1.$$
Now we conclude that $$\delta^{\alpha}\delta^{3}\lesssim \int_{|x|<\delta} |x|^{\alpha}\lesssim \int_{\R^3} |Af(x)|^p |x|^{\alpha} \leq C \int_{\R^3} |f(x)|^p |x|^{\alpha}\simeq \delta^2,$$
and letting $\delta\to 0,$ we obtain that $\alpha+1\geq 0\iff \alpha\geq -1.$ 

For the other side of the inequality one can simply use duality, however, we prefer to present a direct example. We will test the inequality on $f=\chi_{B(0, \delta)}.$ Take $x\in N(\frac{\delta}{2}),$ then there is a $s_{0}\in I$ such that $|x+\gamma(s_0)|<\frac{\delta}{2},$ hence for all $|s-s_{0}|<\frac{\delta}{2},$ using the mean value theorem, we have $|x-\gamma(s)|\leq |x-\gamma(s_0)|+|\gamma(s)-\gamma(s_0)|\lesssim \frac{\delta}{2}+O(|s-s_{0}|)\lesssim \delta.$ Then for $x\in N(\frac{\delta}{2}),$ we have 
$$A f(x)\gtrsim \int_{|s-s_{0}|\leq \delta} f(x-\gamma(s))\chi(s)\, ds\gtrsim \delta.$$ Therefore, 
\begin{eqnarray*}
 \int_{N(\frac{\delta}{2})} |A f(x)|^p |x|^{\alpha} ~dx \gtrsim \delta^p    \int_{N(\frac{\delta}{2})} |x|^{\alpha} ~dx\gtrsim \delta^{p+2}. 
\end{eqnarray*}
Now, from the following inequality 
\[ \delta^{p+2} \lesssim \int |A^{\gamma}f(x)|^p |x|^{\alpha}~dx\lesssim \int |f(x)|^p |x|^{\alpha} ~dx\lesssim \delta^{\alpha+3},\]
letting $\delta\to 0$ we get $\alpha\leq p-1$.
\end{proof}

\begin{proof}[Proof of Theorem~\ref{Thm:nece}, Part(2)] We use a Knapp-type example to conclude the result. From the nondegeneracy condition and using an affine transformation, we may assume that $\gamma^1(s_0)=e_{1},$ $\gamma^2(s_0)=e_{2},$ and $\gamma^3(s_0)=e_{3}$ for some $s_{0}\in I.$ Moreover, we can assume $c:=\gamma_{3}(s_0)>0.$ Set $\delta>0.$ We can define the test function as $f_{\delta}:=\chi_{\{|x_{1}|<\delta,\, |x_2|<\delta^2,\, |x_3|<\delta^3\}},$ then $$\|f_{\delta}\|_{L^p(|x|^{\alpha})}^{p}\lesssim \delta^{\alpha+6}.$$
We begin with computing the maximal function $M^{\gamma}f_{\delta}(x)$ for $x\in P_{\delta},$ where $$P_{\delta}:=\big\{x=(x_1, x_2, x_3): \big|c x_1-x_{3}\gamma_{1}(s_0)\big|<c\delta,\, \big|c x_2-x_{3}\gamma_{2}(s_0)\big|<c\delta^2,\, x_{3}\sim c \big\}.$$ Naturally, $M^{\gamma}f_{\delta}(x)\geq A_{t(x)}^{\gamma} f_{\delta}(x),$ where $t(x):=\frac{x_{3}}{c}\sim 1,$ this in turn implies that 
\begin{align*}
\big|x_{1}-t(x)\gamma_{1}(s)\big|&\lesssim \big|x_{1}-t(x)\gamma_{1}(s_0)\big|+ \big|t(x)\gamma_{1}(s_0)-t(x)\gamma_{1}(s)\big|\\
&\lesssim \delta+ \big|\gamma(s_0)-\gamma_{1}(s)\big|,  
\end{align*}
and by similar computations $\big|x_{2}-t(x)\gamma_{2}(s)\big|\lesssim \delta^2+\big|\gamma(s_0)-\gamma_{1}(s)\big|$ and $\big|x_{3}-t(x)\gamma_{3}(s)\big|\lesssim \big|\gamma_{3}(s_0)-\gamma_{3}(s)\big|$ holds true. Since $\gamma^1(s_0)=e_{1},$ $\gamma^2(s_0)=e_{2},$ and $\gamma^3(s_0)=e_{3},$ an easy application of the Taylor's theorem concludes that 
\begin{align*}
&\gamma(s)-\gamma(s_0)\\
&=\gamma^{1}(s_{0})(s-s_0)+\frac{\gamma^{2}(s_{0})}{2}(s-s_0)^2+\frac{\gamma^3(s_0)}{6}(s-s_0)^3+O(|s-s_0|^4)\\
&=\Big((s-s_0), \frac{(s-s_0)^2}{2}, \frac{(s-s_0)^3}{6}\Big)+O(|s-s_0|^4).
\end{align*}
Therefore, for $|s-s_0|<\delta,$  $|x_{1}-t(x)\gamma_{1}(s)|\lesssim \delta,$ $|x_{2}-t(x)\gamma_{2}(s)|\lesssim \delta^2,$ and $|x_{3}-t(x)\gamma_{3}(s)|\lesssim \delta^3,$ consequently, $f_{\delta}(x-t(x)\gamma(s))=1.$ The above discussion shows that  
\begin{align}
M^{\gamma}f_{\delta}(x)\geq A_{t(x)}^{\gamma} f_{\delta}(x)\geq \int_{|s-s_0|\leq \delta} f_{\delta}(x-t(x)\gamma(s))\chi(s)\, ds\geq \delta,    
\end{align}
for all $x\in P_{\delta}.$ Hence,
\begin{align}
\delta^p\delta^3\simeq \delta^p |P_{\delta}|\lesssim \int |M^\gamma f_{\delta}(x)|^p\, |x|^{\alpha}\lesssim \|f_\delta\|_{L^p(|x|^\alpha)}^p\lesssim \delta^{\alpha+6},    
\end{align}
now letting $\delta\to 0,$ we get $\alpha\leq p-3.$
\end{proof}

The following proposition characterizes $\mathcal{A}_1$ weights for the averaging operators $A_t.$ It also plays a pervasive role in our analysis.  
\begin{proposition}\label{pointwiseestimate}
Let $\gamma$ be a non-degenerate curve satisfying \eqref{nondegenerate} and $\omega_{\alpha}(x)=|x|^{\alpha}.$
Then
$$A_{t}\,\omega_{\alpha}(x)\leq c\, \omega_{\alpha}(x),\quad \text{a.e.},$$
uniformly in $t,$ if and only if $-1< \alpha\leq 0.$
\end{proposition}
\begin{proof}
Since $\gamma$ satisfies \eqref{nondegenerate}, we can always find a subinterval of $I$ such that $K\geq |\gamma(s)|\geq c$ on that subinterval and which we are again calling by $I.$ As usual we will only write for $t=1,$ and we denote $A:=A_1$ as before.  

When $|x|\geq 4K,$ observe that $|x-\gamma(s)|\geq |x|-|\gamma(s)|\gtrsim |x|,$ hence $|x-\gamma(s)|\simeq |x|$ and we conclude $A(\omega_{\alpha})(x)\lesssim |x|^{\alpha}=\omega_{\alpha}(x).$ Similarly, one can handle the case when $|x|\ll c.$ So, we focus on the case when $c\leq |x|\leq 4 K,$ as in this case the quantity $|x-\gamma(s)|$ can take arbitrarily small values. Denote $I_{k}:=\{s: |x-\gamma(s)|\simeq 2^{-k}\},$ and we use a standard dyadic decomposition
\begin{align*}
A (\omega_{\alpha})(x)=\sum_{k\geq 0} \int_{I_{k}} |x-\gamma(s)|^{\alpha}\chi(s)\, ds\simeq \sum_{k\geq 0}  2^{-k\alpha}2^{-k}, 
\end{align*}
where we used that the arc-length measure of $I_{k}$ is $2^{-k.}$ Now, the estimate $A (\omega_{\alpha})(x)\lesssim |x|^\alpha$ holds true, provided the above sum converges and which is true for $\alpha+1>0.$ This completes the proof of sufficient part of Proposition \ref{pointwiseestimate}.

For the necessary part, we proceed as follows. From the nondegeneracy condition, we assume as before that $\gamma^1(s_0)=e_{1},$ $\gamma^2(s_0)=e_{2},$ and $\gamma^3(s_0)=e_{3}$ for some $s_{0}\in I.$ An application of the Taylor's theorem concludes that 
\begin{align}
\nonumber\gamma(s)-\gamma(s_0)&=\gamma^{1}(s_{0})(s-s_0)+\frac{\gamma^{2}(s_{0})}{2}(s-s_0)^2+\frac{\gamma^3(s_0)}{6}(s-s_0)^3+O(|s-s_0|^4)\\
\label{taylor-2}&=\Big((s-s_0), \frac{(s-s_0)^2}{2}, \frac{(s-s_0)^3}{6}\Big)+O(|s-s_0|^4).
\end{align}
for $s$ sufficiently close to $s_0.$ Considering this, let us evaluate $A\omega_{\alpha}$ at the point $\gamma(s_0).$ 

\begin{align*}
         A \omega_{\alpha}(\gamma(s_{0}))&\gtrsim\int_{|s-s_{0}|\ll\delta}\Big((s-s_0)^2+\frac{(s-s_{0})^4}{4}+\frac{(s-s_{0})^6}{36}\Big)^{\alpha/2}~ds\\
         &\simeq\int_{|s-s_{0}|\ll\delta} (s-s_0)^{\alpha} \Big[  1+\frac{1}{4}(s-s_0)^2+\frac{1}{36}(s-s_0)^4\Big]^{\alpha/2}~ds\\
         &\gtrsim  \int_{|s-s_{0}|\ll\delta} (s-s_0)^{\alpha}~ds.
     \end{align*}
Observe that $\int_{|s-s_{0}|\ll\delta} (s-s_0)^{\alpha}~ds$ is finite provided $\alpha> -1$. This completes the proof. 
\end{proof}

We complete this section with the proof of  Proposition~\ref{single}.

\begin{proof}[Proof of  Proposition~\ref{single}]
The proof uses the well-known factorization theorem from \cite{Jawerth}: let $1\leq p<\infty$ and $T$ be a positive linear self-adjoint operator, the weighted inequality $T : L^p(\omega)\to L^p(\omega)$ holds true if and only if there exist two weights $\omega_0, \omega_1$ such that $\omega=\omega_0 \omega_{1}^{1-p}$ and $T\omega_{i}(x)\lesssim \omega_{i}(x)$ a.e. for $i=1, 2.$ Let $-1<\alpha<0,$ set $\omega_{0}=|x|^{\alpha}$ and $\omega_1=1.$ Using Proposition~\ref{pointwiseestimate} we get $A_{t}\omega_{i}\lesssim\omega_{i},$ therefore the abstract factorization theorem concludes that $A_{t}$ maps $L^p(|x|^\alpha)$ to itself. For $0<\alpha<p-1,$ we define $\omega_{0}=1$ and $\omega_1=|x|^{-\frac{\alpha}{p-1}}$ and conclude as before that $A_{t}$ maps $L^p(|x|^\alpha)$ to itself.
For the end-point case $\alpha=-1,p-1$ we proceed as follows.

We can safely assume that for all $s\in I,$ $c\leq |\gamma(s)|\leq K$ for some  constants $c, K>0$. Let $0<\gamma<1,$ we define
\begin{eqnarray*}
    \omega_0(x)&=&|x|^{-\gamma-1}\chi_{\{|x|\leq c/2\}}+ \Big|\frac{c+K}{2}-|x|\Big|^{-\gamma}\chi_{\{\frac{c}{2}\leq |x|\leq 2K\}}+|x|^{-1}\chi_{\{|x|\geq 2K\}},\\
   \omega_1(x)&=&|x|^{\frac{-\gamma}{p-1}}\chi_{\{|x|\leq c/2\}}+ \Big|\frac{c+K}{2}-|x|\Big|^{\frac{-\gamma}{p-1}}\chi_{\{\frac{c}{2}\leq |x|\leq 2K\}}+\chi_{\{|x|\geq 2K\}}.
\end{eqnarray*}
We only need to prove that $A_{t}\omega_{i}\lesssim \omega_{i},$ then the rest follows from the factorization theorem. We will show it for $t=1.$ The first and third term in $\omega_{0}$ can be handled as in the Proposition~\ref{pointwiseestimate}. While estimating the term $A\big(\big|\frac{c+K}{2}-|\cdot|\big|^{-\gamma}\chi_{\{\frac{c}{2}\leq |x|\leq 2K\}}\big)(x)$ singularities arise when $|x|~\simeq \frac{c+K}{2}.$ Now we decomose
\begin{align}
\nonumber &A\big(|\frac{c+K}{2}-|y||^{-\gamma}\chi_{\frac{c}{2}<|y|<2K})(x)\\
\nonumber&\lesssim \sum_{k\geq 0}\int_{I_{k}} \big|\frac{c+K}{2}- |x-\gamma(s)|\big|^{-\gamma} \,ds
\end{align}
where $I_{k}:=\{s\in I: |x-\gamma(s)|\simeq \frac{c+K}{2}-\frac{1}{2^k}\}$ or $I_{k}:=\{s\in I: |x-\gamma(s)|\simeq \frac{c+K}{2}+\frac{1}{2^k}\}.$ Then it is easy to see that
$$A\big(|\frac{c+K}{2}-|y||^{-\gamma}\chi_{\frac{c}{2}<|y|<2K})(x)\lesssim \sum_{k\geq 0} 2^{k\gamma} 2^{-k}\lesssim \omega_{0}(x),$$
for $|x|\simeq \frac{c+K}{2}.$ Similarly, one can prove $A(\omega_1)\lesssim \omega_{1}.$ This concludes the proof for $\alpha=-1,$ by duality one can obtain the result for $\alpha=p-1.$

%{\color{blue} 
%We decompose the interval $I=[-1,1]$ into two sets $I_1$ and $I_2$ such that for $s\in I_1$, $|\gamma(s)|\leq c$,  and for $s\in I_2$, $c\leq |\gamma(s)|\leq K$ for some large constant $K>0$. Depending on this decomposition we decompose the averaging operator $A$ as $A^1+A^2$, and separately prove the desired estimate  
%for both $A^1$ and $A^2$. First we'll deal with $A^2$. Define for some $0<\gamma<1$,
%\begin{eqnarray*}
    %\omega_0(x)&=&|x|^{-\gamma-1}\chi_{\{|x|\leq c/2\}}+ \Big|\frac{c+K}{2}-|x|\Big|^{-\gamma}\chi_{\{\frac{c}{2}\leq |x|\leq 2K\}}+|x|^{-1}\chi_{\{|x|\geq 2K\}},\\
   %\omega_1(x)&=&|x|^{\frac{-\gamma}{p-1}}\chi_{\{|x|\leq c/2\}}+ \Big|\frac{c+K}{2}-|x|\Big|^{\frac{-\gamma}{p-1}}\chi_{\{\frac{c}{2}\leq |x|\leq 2K\}}+\chi_{\{|x|\geq 2K\}}. 
%\end{eqnarray*}
%It can be easily checked that $A^2\omega_i\lesssim \omega_i$ for $i=0,1$. Hence the factorization theorem concludes that $A^2$ is bounded in $L^p(|\cdot|^{-1})$. Now for $A^1$, we define the weights as follows 
%\begin{eqnarray*}
    %\omega_0(x)&=&|x|^{-\gamma}\chi_{\{|x|\leq 2c\}}+ |x|^{-1} \chi_{\{|x|>2c\}}\\
    %\omega_1(x)&=&|x|^{\frac{-\gamma+1}{p-1}}\chi_{\{|x|\leq 2c\}}+ \chi_{\{|x|>2c\}}.
%\end{eqnarray*}
%Therefore, $A^1\omega_i\lesssim \omega_i$ for $i=0,1$. Now the abstract factorization theorem concludes $L^p(|\cdot|^{-1})$ boundedness of $A^1$. By duality, we get $L^{p}(|\cdot|^{p-1})$ boundedness of $A$.}

\end{proof}

\section{Proof of main results}
\label{proof:main}
This section is dedicated to proving our main results. Throughout this section $\gamma: [-1, 1]\to \mathbb{R}^3$ is a smooth non-degenerate curve satisfying \eqref{nondegenerate}. We begin with recalling local smoothing estimates from \cite{BeltranPLMS}.  Let $b: \mathbb{R}^{3}\setminus \{0\}\times \mathbb{R}\times \mathbb{R}$ be a $C^{\infty}$ function, called as symbol, and we denote the corresponding multiplier as
$$m_{\gamma}[b](\xi;t)=\int_{\R} e^{-i\, t \langle \gamma(s),\xi\rangle} b(\xi;t;s)\Psi_{I}(s)\rho(t)\,ds,$$
for some $\Psi_{I}\in C_{c}^{\infty}(\mathbb{R})$ with $\text{supp}{\Psi_{I}}\subset I:=[-1, 1].$ Using a smooth partition of unity we write $b=\sum_{j\geq 0}b_{j}$ as a sum of frequency localized pieces, where $$b_{j}(\xi;t;s)=b(\xi;t;s) \beta_{j}(\xi),$$ 
and $\beta_{j}(\xi):=\beta(|\xi|/2^{j})$ where the smooth function $\beta$ is defined as $\beta(r):=\zeta(r)-\zeta(2r)$ and $\zeta$ is a non-negative, even, compactly supported smooth function with $\zeta=1$ on $I$ and $\zeta=0$ outside $[-2, 2].$ Finally, $b_{0}$ denotes the low-frequency part $b_{0} (\xi;t;s)=b(\xi;t;s) \zeta(|\xi|).$ For simplicity, we denote \[A^j_tf(x)=m_{\gamma}[b_j](D;t)f(x),\]
where $m_{\gamma}[b_j](\xi;t)=\int_{\R} e^{-i\, t \langle \gamma(s),\xi\rangle} b_{j}(\xi;t;s)\Psi_{I}(s)\rho(t)\,ds$, where $b_j'$s are as described above. Now we recall the $L^p-L^q$ local smoothing estimate from \cite{BeltranPLMS}.
\begin{proposition}[Proposition 3.2 \cite{BeltranPLMS}]\label{Prop3.2}
    Let $\gamma$ as in Theorem \ref{full}. Then for $j\in \mathbb{N}$ and for all $\epsilon>0$,
    \[\Big(\int^2_1 \int_{\R^3} |A^{j}_tf(x)|^6 ~dx dt\Big)^{1/6}\lesssim_{\epsilon} 2^{-j/6} 2^{\epsilon j} \Vert f\Vert_{L^{4}(\R^3)}.\]
\end{proposition}

We also recall the sharp $L^p-L^p$ local smoothing estimate recently proved in \cite{Changkeon}. Though the theorem is proved in any dimension $n\geq 2,$ we only write it for $n=3$

\begin{theorem}[Theorem 1.1 \cite{Changkeon}]\label{LS:thm}
Let $\gamma$ be a smooth non-degenerate curve satisfying \eqref{nondegenerate}. Let $2\leq p<\infty$ and denote $\sigma(p, 3):=\min{\{\frac{1}{3}, \frac{1}{3}(\frac{1}{2}+\frac{2}{p}), \frac{2}{p}\}}.$ Then for all $j\in \mathbb{N},$ and for all $\epsilon>0$  we have
\[\Big(\int^2_1 \int_{\R^3} |A^{j}_tf(x)|^p ~dx dt\Big)^{1/p}\lesssim_{\epsilon} 2^{-j\sigma(p, 3)} 2^{\epsilon j} \Vert f\Vert_{L^{p}(\R^3)}.\]
\end{theorem}

Now that we have gathered our prerequisites, let us turn our attention to the proof of Theorem~\ref{full}. Define $$\mathcal{W}_1=\{\omega: M\omega(x)\lesssim \omega(x),~\text{for a.e.} x\}.$$
From Proposition~\ref{pointwiseestimate}, we can easily conclude that $\omega_{\alpha}(x)=|x|^{\alpha}\in \mathcal{W}_1$ for $-1<\alpha\leq 0$. Following \cite{JavierVega}, $M$ is bounded on $\omega L^{\infty}$ if and only if $\omega\in \mathcal{W}_{1},$ where the space $\omega L^{\infty}$ denotes the space of measurable functions $f$ such that $\omega^{-1}\, f\in L^{\infty}$ and $\|f\|_{\omega L^{\infty}}=\|\omega^{-1}\, f\|_{L^{\infty}}.$ Therefore, for all $\omega_{\alpha}(x)=|x|^{\alpha}$ with $-1<\alpha\leq 0,$ $M$ maps $\omega_{\alpha}L^{\infty}$ to itself. Now, applying interpolation with change of measure with the unweighted estimate $M: L^p\to L^p$ for $3<p\leq \infty$ we obtain that:
\[M: L^p(\omega_{\alpha}^{p_{0}-p})\to L^p(\omega_{\alpha}^{p_{0}-p})\]
for $3<p_{0}<p$ and $-1<\alpha\leq 0.$ This implies the following result: \textit{$M$ is bounded on $L^{p}(|x|^{\alpha})$, for $3<p<\infty$  and $0\leq \alpha<p-3$.} This concludes Theorem~\ref{full} for $0\leq \alpha<p-3.$ The proof of $L^p(|x|^{\alpha})$ boundedness for negative values of $\alpha$ is more challenging and will be divided into subsequent steps.

\subsection{Proof of Theorem~\ref{full} for $4\leq p<\infty$}

Using \ref{Prop3.2} we first prove the following estimate at $p=4$ which will enable us to obtain the complete range of weighted estimates for $4\leq p<\infty.$
\begin{lemma}\label{deltagainfor4}
    Let $0<\delta<1$ and $j\in\mathbb{N}$. Then for any $\epsilon>0$ we have
     \[\Big( \int_{|x|<\delta} \sup_{1\leq t\leq2}|A^{j}_tf(x)|^4 ~dx \Big)^{1/4}\lesssim_{\epsilon} \delta^{1/4} 2^{\epsilon j} \Vert f\Vert_{L^{4}(\R^3)}.\]
\end{lemma}

\begin{proof}
    Using H\"{o}lder's inequality we get
    \begin{eqnarray*}
        \Big( \int_{|x|<\delta} \sup_{1\leq t\leq2}|A^{j}_tf(x)|^4 ~dx \Big)^{1/4} &\lesssim&  \delta^{1/4} \Big( \int_{|x|<\delta} \sup_{1\leq t\leq2}|A^{j}_tf(x)|^6 ~dx \Big)^{1/6}\\
        &\lesssim& \delta^{1/4} \Big(\int_{|x|<\delta} \int^2_{1}|\partial^{\beta}_{t}(A^{j}_{t})f(x)|^6~ dt dx\Big)^{1/6}.
    \end{eqnarray*}
    In the last inequality, we have used the one-dimensional Sobolev embedding with $\beta>1/6$. To estimate the above, we interpolate between $\beta=0$ and $\beta=1$. Indeed, using Proposition \ref{Prop3.2} we get that for any $\epsilon>0$
     \begin{eqnarray*}
         \Big(\int^2_1 \int_{|x|<\delta} |A^{j}_tf(x)|^6 ~dx dt\Big)^{1/6}
        &\lesssim_{\epsilon}&  2^{-j/6}2^{j\epsilon} \Vert f\Vert_{L^4(\R^3)}.
    \end{eqnarray*}
    On the other hand, as the operator $\partial_{t}(A^{j}_t)$ is essentially same as $2^{j}A^{j}_t,$ we have
    \begin{eqnarray*}
         \Big(\int^2_1 \int_{|x|<\delta} |\partial_{t}(A^{j}_t)f(x)|^6 ~dx dt\Big)^{1/6}
        &\lesssim_{\epsilon}& 2^j 2^{-j/6}2^{j\epsilon} \Vert f\Vert_{L^4(\R^3)}.
    \end{eqnarray*}
    Hence, combining the above two estimates we get
    \begin{eqnarray*}
         \Big(\int^2_1 \int_{|x|<\delta} |\partial^{\beta}_{t}(A^{j}_t)f(x)|^6 ~dx dt\Big)^{1/6}
        &\lesssim_{\epsilon}& 2^{\beta j} 2^{-j/6}2^{j\epsilon} \Vert f\Vert_{L^4(\R^3)},
    \end{eqnarray*}
    where $\beta>1/6$. Since $\epsilon>0$ is arbitrarily small, choosing $\beta>\frac{1}{6}$ sufficiently close to $\frac{1}{6}$ we get $2^{\beta j} 2^{-j/6}2^{j\epsilon}\leq 2^{2j\epsilon}$. This completes the proof.
\end{proof}

\begin{proof}[Proof of Theorem~\ref{full}, Part (1)]
{The proof of Theorem~\ref{full}, Part (1) will be done in several steps. We first prove the following estimate for the local maximal operator, that is, we shall prove that there exists some $\epsilon(p)>0$ such that
\begin{align}
\label{local}
\int_{\R^3} \sup_{1\leq t\leq2}|A^{j}_{t}f(x)|^p |x|^{\alpha} ~dx\leq C 2^{-\epsilon(p)j} \int_{\R^3} |f(x)|^p |x|^{\alpha}~dx,
\end{align}
for $4\leq p<\infty$ and $-1<\alpha<0.$ Once the proof of \eqref{local} is achieved, the passage from local to global is standard, we provide a brief outline. Recall that we can write $$Mf(x)\lesssim \sum_{j=1}^{\infty}\big(\sum_{\ell\in \Z}\sup_{1\leq t\leq 2}|A_{2^{\ell}t}\beta_{j-l}(D)f(x)|^{p}\big)^{1/p}+M_{\text{HL}}f(x),$$
where $M_{\text{HL}}$ stands for the Hardy-Littlewood maximal operator. By triangle inequality, $\|Mf\|_{L^p(|x|^{\alpha})}$ is dominated by
\begin{align}
\sum_{j=1}^{\infty}\big(\sum_{\ell\in \Z} \|\sup_{1\leq t\leq 2}|A_{2^{\ell}t}\beta_{j-l}(D)f\|_{L^p(|x|^{\alpha})}^{p}\big)^{1/p}.
\label{loc-to-glo}
\end{align}
Now a simple rescaling argument shows that 
$$\sup_{1\leq t\leq 2 } \big|A_{2^{\ell}t}\beta_{j-l}(D)f(x)\big|=\sup_{1\leq t\leq 2} \big|A_{t}\beta_{j}(D)f_{2^{\ell}}\big(\frac{x}{2^{\ell}}\big)\big|,$$ 
where $f_{2^{\ell}}(\cdot)=f(2^{\ell}\cdot).$ As a consequence, we obtain \[\|\sup_{1\leq t\leq 2}|A_{2^{\ell}t}\beta_{j-l}(D)\|_{L^p(|x|^{\alpha})\to L^p(|x|^{\alpha})}=\|\sup_{1\leq t\leq 2}A^{j}_{t}{\beta}_{j}(D)\|_{L^p(|x|^{\alpha})\to L^p(|x|^{\alpha})}.\]
Now using \eqref{local} in \eqref{loc-to-glo} we obtain that $\|Mf\|_{L^p(|x|^{\alpha})}$ is dominated by 
$$\sum_{j=1}^{\infty} 2^{-\epsilon(p)j} (\sum_{\ell\in \Z} \|\tilde{\beta}_{j-l}(D)(f)\|_{L^p(|x|^{\alpha})}^{p}\big)^{1/p},$$
which, after summing in $j,$ can be further dominated by $$\|(\sum_{\ell\in \Z} |\tilde{\beta}_{l}(D)(f)|^2)^{1/2}\|_{L^p(|x|^{\alpha})}\leq c\|f\|_{L^p(|x|^{\alpha})}$$ as $2\leq p<\infty$ and $|x|^{\alpha}$ is an $\mathcal{A}_{p}$ weight on $\mathbb{R}^3$ for $-1<\alpha\leq 0.$}

Here onwards we focus on proving \eqref{local}. The proof of \eqref{local} crucially depends on the following estimate. For $j\in \mathbb{N}$ and $\epsilon>0$, we have
\begin{align}
\label{mainprop4}
\int_{\R^3} \sup_{1\leq t\leq2} |A^j_tf(x)|^4 |x|^{-1} ~dx\lesssim_{\epsilon} 2^{j\epsilon} \int_{\R^3} |f(x)|^4  |x|^{-1}~dx.
\end{align}
Once \eqref{mainprop4} is proved, Stein--Weiss interpolation with the unweighted estimate  
\begin{align}
\label{decayestimate}
\int_{\R^3} \sup_{1\leq t\leq2} |A^j_tf(x)|^p ~dx\lesssim 2^{-j\gamma(p)} \int_{\R^3} |f(x)|^p\,dx\, ;\quad \gamma(p)>0, 3<p< \infty
\end{align}
yields that there exists some $\gamma'(\alpha, p)>0$ such that
\begin{align}
\label{p:large}
\int_{\R^3} \sup_{1\leq t\leq2}|A_{t}^{j}f(x)|^p |x|^{\alpha} ~dx\lesssim 2^{-j \gamma'(\alpha, p)}\int_{\R^3} |f(x)|^p |x|^{\alpha}~dx    
\end{align}
for $4\leq p<\infty$ and $-1<\alpha<0.$ Summing \eqref{p:large} over $j\geq 0,$ we obtain \eqref{local}. Now the rest of the proof is dedicated to prove \eqref{mainprop4}.

As $\gamma$ satisfies \eqref{nondegenerate}, there is no harm in assuming that $0<c\leq |\gamma(s)|\leq K$ on some subinterval of $I,$ and we call the subinterval as $I$ again. By rescaling we consider $c=2^{-1}$ and $K=2^2$. Observe that, by method of stationary phase, it is enough to work with the main term of $A^{j}_{t}f$, which we denote by $$\widetilde{A}^{j}_{t}f(x)=2^{2j} \chi_{\{||\cdot|-t|\gamma (\cdot)||<t2^{-j}\}}\ast f(x),$$ as the remaining terms decay rapidly. This localization helps us to focus on the most singular part of the operator. Since $t\in[1,2]$, we observe that when $supp(f)\subset \{|x|\leq 1/4\}$ then $supp(\widetilde{A}^{j}_{t}f)\subset \{1/4\leq |x|\leq 2^4\}$, and $|x|^{-1}$ behaves like a constant in $\{1/4\leq |x|\leq 2^4\}$. Moreover, when
$supp(f)\subset \{2^k\leq |x|\leq 2^{k+1}\}$ for $k\geq5$, then $supp(\widetilde{A}^{j}_{t}f)\subset \{2^k-16\leq |x|\leq 2^k+16\}$. In this case the weight $|x|^{-1}$ behaves like $2^{-k}$ in both the sides. Therefore, the desired weighted estimate follows from the unweighted estimate of the operator $A^j_{t}$. Therefore, we proceed with the proof of \eqref{mainprop4} when $supp(f)\subset \{2^{-2}\leq |x|\leq 2^5\}$. We decompose the integral the right hand side of \eqref{mainprop4} as following
\begin{eqnarray*}
    &&\int_{\R^3} \sup_{1\leq t\leq2} |\widetilde{A}^{j}_{t}f(x)|^4 |x|^{-1}~dx\\
    &&= \int_{|x|\leq 2^{-j}} \sup_{1\leq t\leq2} |\widetilde{A}^{j}_{t}f(x)|^4 |x|^{-1}~dx+\sum^{j}_{k=-4}\,\, \int\limits_{2^{-k}\leq |x|<2^{-k+1}} \sup_{1\leq t\leq2} |\widetilde{A}^{j}_{t}f(x)|^4 |x|^{-1}~dx \\
    &&:= A+B.
\end{eqnarray*}
We first deal with $B$. Using the Lemma~\ref{deltagainfor4} with $\delta=2^{-k+1},$ we obtain
\begin{align*}
    B&\lesssim \sum^{j}_{k=-4} 2^{k}\, \int\limits_{2^{-k}\leq |x|<2^{-k+1}}\, \sup_{1\leq t\leq2} |\widetilde{A}^{j}_{t}f(x)|^4 ~dx\\
    &\overset{\eqref{deltagainfor4}}{\lesssim} \sum^{j}_{k=-4} 2^{k} 2^{-k+1} 2^{j\epsilon} \int |f(x)|^4 ~dx\\
    &\lesssim j^2\,2^{\epsilon j}  \int |f(x)|^4 ~dx\leq 2^{\epsilon j} \int |f(x)|^4 ~dx,~\text{for any}~\epsilon>0.
\end{align*}

Now we proceed with estimating $A$. Here we need to employ the one-dimensional Sobolev embedding theorem. We decompose the function $f$ as $$f=\sum\limits^{2^{j+2}}_{l=0}f_l,$$ where $f_l=f \chi_{B^j_l}$ and $B^j_l=\{x: \frac{1}{2}+l2^{-j}\leq |x|<\frac{1}{2}+(l+1)2^{-j}\}$. Observe that for $|x|\leq 2^{-j}$, $\widetilde{A}^{j}_tf(x)\neq0$ if $t\in I_l$, where length of $I_l$ is equivalent to $2^{-j}$. Therefore,
\begin{align}
  \nonumber  \int_{|x|\leq 2^{-j}} \int^2_1 |\widetilde{A}^{j}_tf(x)|^4 |x|^{-1} ~dt dx &=    \int^2_1\int_{|x|\leq 2^{-j}} |\widetilde{A}^{j}_tf(x)|^4 |x|^{-1} ~dx dt\\
    \nonumber &\leq\sum^{2^{j+2}}_{l=0} \int_{I_l} \int_{|x|\leq 2^{-j}} |\widetilde{A}^{j}_tf_l(x)|^4 |x|^{-1} ~dx dt\\
    \nonumber& \lesssim \sum^{2^{j+2}}_{l=0} 2^{-j} \int |f_l(x)|^4 |x|^{-1}~dx\\
    &\lesssim 2^{-j}  \int |f(x)|^4 |x|^{-1}~dx. \label{interpol:1}
\end{align}
In the second inequality we have used the fact that the single scale averages $A_{t}$ are bounded from $L^4(|x|^{-1})$ to itself and the constant is uniform in $t,$ see Proposition~\ref{single}.  The final inequality follows since the support of $f_l'$s are disjoint. Recall that the operator $\partial_{t}(\widetilde{A}^{j}_t)\simeq 2^{j} \widetilde{A}^{j}_t,$ hence we obtain
\begin{align}
\int_{|x|\leq 2^{-j}} \int^2_1 |\partial_{t}(\widetilde{A}^{j}_t) f(x)|^4 |x|^{-1} ~dt dx\lesssim 2^{-j} 2^{4j} \int |f(x)|^4 |x|^{-1}~dx.
\label{interpol:2}
\end{align}
Interpolating \eqref{interpol:1} and \eqref{interpol:2} we derive

\begin{align*} \int_{|x|\leq 2^{-j}} \int^2_1 \Big|\partial_{t}^{\beta}\widetilde{A}^{j}_tf(x)\Big|^4 |x|^{-1}\,dt\,dx&\lesssim 2^{-j} 2^{4\beta j} \int |f(x)|^4 |x|^{-1}~dx\\
&\leq 2^{j\epsilon} \int |f(x)|^4 |x|^{-1}~dx,\end{align*}
provided we choose $\beta>\frac{1}{4},$ arbitrarily close to $\frac{1}{4}.$ Now, employing the one-dimensional Sobolev embedding theorem we get
\[A\lesssim 2^{j\epsilon} \int |f(x)|^4 |x|^{-1}~dx.\]
Hence, combining the estimates of $A$ and $B$ we get the desired estimate \eqref{mainprop4}. This completes the proof of \eqref{mainprop4} and thus completes the proof of \eqref{local}.
\end{proof}

\subsection{Proof of Theorem~\ref{full} for $3<p<4$} Our main result in this section reads as follows: for $j\in\mathbb{N}$ and any $\epsilon>0$, we have
\begin{align}
\label{main:three}
\int_{\R^3} \sup_{1\leq t\leq2} |A^j_tf(x)|^3 |x|^{-\frac{1}{2}} ~dx\lesssim_{\epsilon} 2^{j\epsilon} \int_{\R^3} |f(x)|^3  |x|^{-\frac{1}{2}}~dx.
\end{align}
Once this is proved, the results for $3<p<4$ follow from the Stein--Weiss interpolation of \eqref{mainprop4} and \eqref{main:three}. The remaining part of the section is dedicated to prove \eqref{main:three} and, as usual, we will prove the estimate for $\widetilde{A}_{t}^{j}$ and {functions will be assumed to be supported on $\{2^{-2}\leq |x|\leq 2^{5}\}$}. To begin with, we record the following localization estimate at $L^3$ near the origin.  

\begin{lemma}\label{deltagainfor3}
    Let $j\in\mathbb{N}$ and $\delta\geq2^{-j}$. Then we have
    \begin{eqnarray}
        \int_{|x|\leq\delta}\sup_{1\leq t\leq2} |\widetilde{A}^{j}_tf(x)|^3 ~dx\lesssim \delta^{1/2} 2^{\epsilon j} \int |f(x)|^3~dx,
    \end{eqnarray}
    for any $\epsilon>0$.
\end{lemma}

\begin{proof}[Proof of Lemma \ref{deltagainfor3}]
     Our idea is to employ interpolation between $L^2$ estimate and $L^4$ estimate.Let $f_l=f \chi_{B_l(\delta)}$ and $B_l(\delta)=\{x: \frac{1}{2}+l\delta\leq |x|<\frac{1}{2}+(l+1)\delta\},$ then we can write
   \begin{align}
        \nonumber\int_{|x|\leq\delta} \int^2_1|\widetilde{A}^{j}_tf(x)|^2 ~dt dx&\leq \sum^{[\delta^{-1}]+3}_{l=0} \int_{I_l}\int_{|x|\leq\delta} |\widetilde{A}^{j}_tf_l(x)|^2 dx dt\\
        \nonumber&\leq 2^{-2j/3} \sum^{[\delta^{-1}]+3}_{l=0} \int_{I_l} \int_{\R^3} |f_l(x)|^2~dx dt\\
        \label{decay:delta:gain}&\lesssim \delta 2^{-2j/3} \Vert f\Vert^2_{L^2}.
    \end{align}
    It can be checked that for $t\notin I_l$ ( where $|I_l|\lesssim \delta$ and $I_l$'s are disjoint), $\widetilde{A}^{j}_t f_l(x)=0$. In the second inequality we have used Plancherel's identity and that $|\mu_{t}(\xi)|\lesssim (1+t|\xi|)^{-1/3}$. On the other hand, from the sharp $L^p-L^p$ local smoothing estimate Theorem~\ref{LS:thm} we get
   \begin{eqnarray}
   \label{decay:delta:ls}
       \Big(\int_{\R^3} \int^2_1 |A^j_tf(x)|^4 ~dt dx\Big)^{1/4}\lesssim 2^{-\frac{j}{3}+\epsilon j} \Vert f\Vert_{L^4},
   \end{eqnarray}
   for arbitrarily small $\epsilon>0$. Now we apply interpolation between \eqref{decay:delta:gain} and \eqref{decay:delta:ls} to obtain
    \begin{eqnarray*}
        \Big(\int_{|x|\leq\delta}\int^2_{1} |\widetilde{A}^{j}_tf(x)|^3 ~dt dx\Big)^{1/3}\lesssim \delta^{\frac{1}{6}} 2^{-\frac{j}{3}+j\epsilon} \Big(\int |f(x)|^3~dx\Big)^{1/3},
    \end{eqnarray*}
    for arbitrarily small $\epsilon>0$. Finally, employing the one-dimensional Sobolev embedding theorem we get for any $\beta>\frac{1}{3}$ that
    \begin{align*}
        \Big(\int_{|x|\leq\delta} \sup_{1\leq t\leq2} |A^j_tf(x)|^3 ~dx\Big)^{1/3}&\lesssim \Big(\int^2_1 \int_{|x|<\delta} |\partial^{\beta}_{t}(A^{j}_t)f(x)|^3 ~dx dt\Big)^{1/3}\\
        &\lesssim \delta^{1/6} 2^{\beta j} 2^{-j/3} 2^{j\epsilon} \Vert f\Vert_{L^3},
    \end{align*}
for arbitrarily small $\epsilon>0$. Choosing $\beta>\frac{1}{3}$ arbitrarily close to $\frac{1}{3}$ completes the Lemma.
\end{proof}
Now we proceed with the proof of \eqref{main:three}.

\begin{proof}[Proof of \eqref{main:three}]
From similar considerations as in the proof of estimate \eqref{mainprop4}, it is enough to prove \eqref{main:three} for functions supported in $\{2^{-2}\leq |x|\leq 2^5\}$. Then

\begin{align*}
    &\int_{\R^3} \sup_{1\leq t\leq2} |\widetilde{A}^{j}_{t}f(x)|^3 |x|^{-1/2}~ dx\\
    &\leq \int_{|x|\leq 2^{-j}} \sup_{1\leq t\leq2} |\widetilde{A}^{j}_{t}f(x)|^3 |x|^{-1/2}~dx\\
    &+\sum^{j}_{k=-4} \int_{2^{-k}\leq |x|<2^{-k+1}} \sup_{1\leq t\leq2}|\widetilde{A}^{j}_{t}f(x)|^3 |x|^{-1/2}~ dx \\
    &:= A+B.
\end{align*}
We first deal with $B$.
\begin{eqnarray*}
    B&\lesssim& \sum^{j}_{k=-4} 2^{k/2} \int_{2^{-k}\leq |x|<2^{-k+1}} \sup_{1\leq t\leq2}|\widetilde{A}^{j}_{t}f(x)|^3 ~ dx\\
    &\lesssim& \sum^{j}_{k=-2} 2^{k/2} 2^{-k/2} 2^{j\epsilon} \int |f(x)|^3 ~dx\\
    &\lesssim& j^2\,2^{\epsilon j}  \int |f(x)|^3 ~dx\leq 2^{\epsilon j} \int |f(x)|^3 ~dx,~\text{for any}~\epsilon>0.
\end{eqnarray*}
In the second inequality we have applied Lemma \ref{deltagainfor3} with $\delta=2^{-k+1}$.
Now we work with $A$. Here we need to employ the Sobolev embedding theorem. We decompose the function $f$ as $f=\sum^{2^{j+2}}_{l=0}f_l$, where $f_l=f \chi_{B^j_l}$ and $B^j_l=\{x: \frac{1}{2}+l2^{-j}\leq |x|<\frac{1}{2}+(l+1)2^{-j}\}$. Observe that for $|x|\leq 2^{-j}$, $\widetilde{A}^{j}_tf(x)\neq0$ if $t\in I_l$ for finitely many $l$ (at most $8$), where length of $I_l$ is $c 2^{-j}$ for some small constant $c>0$. Therefore,
\begin{eqnarray*}
    \int_{|x|\leq 2^{-j}} \int^2_1 |\widetilde{A}^{j}_tf(x)|^3 |x|^{-\frac{1}{2}} ~dt dx &=&    \int^2_1\int_{|x|\leq 2^{-j}} |\widetilde{A}^{j}_tf(x)|^3 |x|^{-\frac{1}{2}}\,dx dt\\
    &\lesssim&\sum^{2^{j+2}}_{l=0} \int_{I_l} \int_{|x|\leq 2^{-j}} |\widetilde{A}^{j}_tf_l(x)|^3 |x|^{-\frac{1}{2}}\,dx dt\\
    &\lesssim& \sum^{2^{j+2}}_{l=0} c2^{-j} \int |f_l(x)|^3 |x|^{-\frac{1}{2}}\,dx\\
    &\lesssim& 2^{-j}  \int |f(x)|^3 |x|^{-\frac{1}{2}}\,dx.
\end{eqnarray*}
As before, the second inequality follows as $A_{t}$ maps $L^3(|x|^{-1/2})$ to itself, see Proposition~\ref{single} with constant independent of $t$. We can sum as the support of $f_l'$s are disjoint. Similarly, we get
\[ \int_{|x|\leq 2^{-j}} \int^2_1 |(\partial_{t}\widetilde{A}^{j}_t)f(x)|^3 |x|^{-1/2} ~dt dx\lesssim 2^{-j} 2^{3j} \int |f(x)|^3 |x|^{-1/2}~dx.\]
Interpolating the above, choosing $\beta>\frac{1}{3}$ sufficiently close $\frac{1}{3},$ we get
\begin{align*} \int_{|x|\leq 2^{-j}} \int^2_1 |(\partial_{t}^{\beta}\widetilde{A}^{j}_t)f(x)|^3 |x|^{-1/2}\,dt dx&\lesssim 2^{-j} 2^{3\beta j} \int |f(x)|^3 |x|^{-1/2}\,dx\\
&\lesssim 2^{j\epsilon} \int |f(x)|^3 |x|^{-1/2}\,dx.\end{align*}
Then the one-dimensional Sobolev embedding concludes that
\[A\lesssim 2^{j\epsilon} \int |f(x)|^3 |x|^{-1/2}\,dx.\]

Hence, combining the estimates of $A$ and $B$ we get the desired estimate.
\end{proof}

\begin{remark}
We would like to remark that we can even prove the improved estimate
$$A\lesssim 2^{j\epsilon} \int |f(x)|^3 |x|^{-1}\,dx,$$
hence one can improve \eqref{main:three} to the weight $|x|^{-1}$ provided it can be done for $B.$ 
\end{remark}

\subsection{Lacunary maximal function} This section is dedicated in proving the Theorem~\ref{lacunary}. The proof follows the arguments in \cite{JavierVega}, however we give complete details for the readers' convenience. Recall the Littlewood-Paley projections $\widehat{\beta_{j}(D)f}(\xi)=\beta_{j}(\xi) \widehat{f}(\xi),$ where $\beta_{j}(\xi):=\beta(|\xi|/2^{j})$ and the smooth function $\beta$ is defined as $\beta(r):=\zeta(r)-\zeta(2r)$ and $\zeta$ is a non-negative, even, compactly supported smooth function with $\zeta=1$ on $I$ and $\zeta=0$ outside $[-2, 2].$

\begin{proposition}
\label{decay-prop}
Let $1<p<\infty.$ Then there exists $\delta(p)>0$ such that
\begin{align}
\label{unweighted-decay}
\|\sup_{\ell} |A_{2^{\ell}}\beta_{k-\ell}(D)f|\|_{p}\lesssim 2^{-\delta(p) k} \|f\|_{p}, 
\end{align}
for all $k\geq 1.$
\end{proposition}

\begin{proof}
The proof of the proposition requires the square function
$$\mathfrak{G}(f)(x)=\left(\sum_{\ell} |A_{2^{\ell}}\beta_{k-\ell}(D)f(x)|^2\right)^{\frac{1}{2}},$$
initially defined for $f\in \mathcal{S}(\mathbb{R}^3).$ It follows from the decay $|m_{\gamma}(\xi, t)|\leq (1+t|\xi|)^{-1/3}$ and orthogonality that $\|\mathfrak{G}(f)\|_{L^2}^{2}\lesssim \sum_{\ell}\|A_{2^\ell}\beta_{k-\ell}(\tilde{\beta}_{k-l}(D)f)\|^2_{L^2} \lesssim \sum_{\ell} 2^{-\frac{2k}{3}} \|\tilde{\beta}_{k-\ell}(D)f\|^2_{2}\lesssim 2^{-\frac{2k}{3}} \|f\|_{2}^2,$ here $\tilde{\beta}_{k}$ is slight flattening of $\beta_{k}$ such that $\beta_{k}\cdot \tilde{\beta}_{k}=\beta_{k}.$ Therefore, $\|\mathfrak{G}f\|_{L^2}\leq C 2^{-k/3}\|f\|_{L^2}.$

Next we shall prove that $\|\mathfrak{G}f\|_{L^{1, \infty}}\leq C 2^{\epsilon\,k}\|f\|_{L^1}$ for all $\epsilon>0.$ Using a standard randomization procedure, we consider the linearized square $\sum_{\ell} a_{\ell} A_{2^{\ell}}\beta_{k-\ell}(D)f,$ where $a=(a_{\ell})$ is a sequence of IID random variables with each $a_{\ell}$ taking values $\pm 1$ with equal probability, and it is enough to prove that $\|\sum_{\ell}a_{\ell} A_{2^{\ell}}\beta_{k-\ell}(D)f\|_{L^{1, \infty}}\leq C 2^{\epsilon k}\|f\|_{L^1},$ with constant independent of the random variable $a.$ To that end, using Calder\'on-Zygmund decomposition, our task is reduced to proving
\begin{align}
\label{CZ:1}
 \int_{|x|>2|y|}\big|L_{k}(x-y)-L_{k}(x)\big|\, dx\leq C 2^{\epsilon k},  
\end{align}
for all $y\in \mathbb{R}^3,$ and here $L_{k}=\sum_{\ell} a_{\ell} L_{\ell, k}$ with $A_{2^\ell}\beta_{k-\ell}(D)f(x)=L_{\ell, k}\ast f(x).$ Hence, \eqref{CZ:1} will be implied if we prove
\begin{align}
\label{CZ:2}
 \sum_{\ell}\int_{|x|>2|y|}\big|L_{\ell, k}(x-y)-L_{\ell, k}(x)\big|\, dx\leq C 2^{\epsilon k}.  
\end{align}
Using the change of variables $x\to 2^{\ell}x,$ we obtain
$J_{\ell, k}(y)=\int_{|x|>2|y|}\big|L_{\ell, k}(x-y)-L_{\ell, k}(x)\big|\, dx=\int_{|x|>2^{1-\ell}|y|}\big|(\mu\ast \widecheck{(\beta_{k})})(x-2^{-\ell}y)-\mu\ast \widecheck{(\beta_{k})}(x)\big|\, dx.$ Further, replacing $y$ by $2y,$  $J_{\ell, k}$ becomes $J_{\ell-1, k},$ and thus it is enough to estimate $J_{\ell, k}(y)$ for all $\ell\in \mathbb{Z}$ but uniformly for $1\leq |y|\leq 2.$ By Young's inequality, we obtain $J_{\ell, k}(y)\lesssim \|\mu\|\|\widecheck{\beta_{k}}\|_{L^1}\lesssim 1$ uniformly in $y.$ The second estimate for $J_{\ell, k}$ follows from the mean value theorem:
\begin{align}
\nonumber J_{\ell, k}(y)&\lesssim \int\int \frac{1}{2^{(\ell-k)n}}\bigg|\widecheck{\beta}\big(\frac{x-y-2^{\ell}z}{2^{\ell-k}}\big)-\widecheck{\beta}\big(\frac{x-2^{\ell}z}{2^{\ell-k}}\big)\big|\, d\mu(z)\, dx\\
&\lesssim\int \int 2^{k-\ell}|y| \frac{1}{2^{(\ell-k)n}} |\nabla \widecheck{\beta} (\frac{(x-2^{\ell}z)-\theta\,y}{2^{\ell-k}})|\, dx\, d\mu(z)\lesssim  2^{k-\ell},
\end{align}
where the inequality holds as integral of $\frac{1}{2^{(\ell-k)n}} \nabla (\widecheck{\beta}(\cdot))$ is bounded. Finally, we obtain
$$J_{\ell, k}(y)\lesssim \min\{1, 2^{k-\ell}\}.$$
Now \eqref{CZ:2} follows as for $\ell>k,$ we use the estimate $J_{\ell, k}\lesssim 2^{k-\ell}$ to sum the terms, and for each $1\leq \ell\leq k,$ using the trivial bound we get that $\sum_{1\leq \ell\leq k} J_{\ell, k}$ is bounded by $O(k),$ finally for $\ell<0,$ we can use another trivial bound $J_{\ell, k}\lesssim 2^{-(k-\ell)N}$ obtained from the Schwartz decay of $\widecheck{\beta}.$ Eventually, we prove that the RHS of \eqref{CZ:2} is controlled by $O(k)$ which implies \eqref{CZ:2} and thus we have proved $\|\sum_{\ell}a_{\ell} A_{2^{\ell}}\beta_{k-\ell}(D)f\|_{L^{1, \infty}}\leq C 2^{\epsilon k}\|f\|_{L^1}$ for all $\epsilon>0.$ Now interpolating with the decay estimate at $L^2,$ we conclude the theorem for $1< p\leq 2.$ Now, for $2<p<\infty,$ we interpolate with the trivial bound at $L^{\infty}.$ 
\end{proof}

\begin{proposition}
\label{lac-uniform}
 Let $1<p<\infty$ and $\omega\in \mathcal{A}_{p}.$ If there exists a constant $C,$ independent of $\ell,$ such that 
\begin{align}
\label{uniform-est}
\int |A_{2^\ell} f(x)|^p\, \omega(x)\, dx\leq C \int |f(x)|^p\, \omega(x)\, dx     
\end{align}
holds for all $\ell\in\mathbb Z,$ then $M_{\text{lac}}: L^p(\omega^s) \to L^p(\omega^s)$  for all $0\leq s<1.$   
\end{proposition}

\begin{proof}
Let $1<p<2$ and $0\leq s<1$. For $N\in \mathbb{N},$ let $\Lambda(N)$ be the smallest constant such that the following inequality holds true:
\begin{align}
\label{Christ-est}
\big\|\sup_{\ell\in \mathbb{Z}: |\ell|\leq N} |A_{2^\ell}f|\big\|_{L^p(\omega^s)}\leq \Lambda(N)\, \big\|f\big\|_{L^p(\omega^s)}. 
\end{align}
Our goal is to prove that $\Lambda(N)\lesssim 1,$ uniformly in $N$. The uniform estimate \eqref{uniform-est} and \eqref{Christ-est} implies that 
\begin{align}
\label{vv-1}
\big\|\sup_{\ell\in \mathbb{Z}: |\ell|\leq N} |A_{2^\ell} \beta_{k-l}(D) g_{\ell}|\|_{L^p(\omega^s)}\leq \Lambda(N)\, \big\|\sup_{\ell\in \mathbb{Z}: |\ell|\leq N} |g_{\ell}|\big\|_{L^p(\omega^s)}.  
\end{align}
To elaborate, using the facts that the operators $A_{2^l}$ are positive and $|\beta_{k-l}(D)g_{\ell}(x)|\lesssim M_{\text{HL}}g_{\ell}(x)\lesssim M_{\text{HL}}(\sup_{|l|\leq N} g_{\ell})(x),$ we obtain
\begin{align}
&\nonumber\big\|\sup_{|\ell|\leq N}  |A_{2^\ell} \beta_{k-l}(D) g_{\ell}|   \|_{L^p(\omega^s)}\\
&\nonumber \lesssim \big\|\sup_{|\ell|\leq N}  |A_{2^\ell}( M_{\text{HL}}(\sup\limits_{|\ell|\leq N} g_{\ell}))|   \|_{L^p(\omega^s)}\\
&\overset{\eqref{Christ-est}}{\lesssim} \Lambda(N)\, \big\| M_{\text{HL}}(\sup_{|\ell|\leq N}  \left| g_{\ell}  \right| )\|_{L^p(\omega^s)}\lesssim \Lambda(N)\big\|\sup_{|\ell|\leq N}  \left| g_{\ell}  \right|\|_{L^p(\omega^s)}.\label{vv-interpol-1}
\end{align}
The last inequality uses that the Hardy--Littlewood maximal operator $M_{\text{HL}}$ maps $L^p(\omega^s)$ to itself for $\omega\in \mathcal{A}_{p}$ and $1<p<\infty.$ Moreover, the uniform estimate \eqref{uniform-est}  yields
\begin{align}
\nonumber&\Big\|\Big(\sum_{|\ell|\leq N} |A_{2^\ell} \beta_{k-l}(D) g_{\ell}|^{p}\Big)^{\frac{1}{p}}\Big\|_{L^p(\omega)}^{p}\\
\nonumber & = \sum_{|\ell|\leq N} \big\|A_{2^\ell} \beta_{k-l}(D) g_{\ell}\big\|_{L^p(\omega)}^{p}\\
&\overset{\eqref{uniform-est}}{\leq} C\, \sum_{|\ell|\leq N} \big\| \beta_{k-l}(D) g_{\ell} \big\|_{L^p(\omega)}^{p}   \lesssim \Big\|\Big(\sum_{|\ell|\leq N} |g_{\ell}|^{p}\Big)^{\frac{1}{p}}\Big\|_{L^p(\omega)}^{p}. \label{vv-interpol-2} 
\end{align}
Interpolation of \eqref{vv-interpol-1} and \eqref{vv-interpol-2} derives the following crucial vector-valued estimate
\begin{align}
\label{vv-interpol-3}
& \Big\|\Big(\sum_{|\ell|\leq N} |A_{2^\ell} \beta_{k-l}(D) g_{\ell}|^{2}\Big)^{\frac{1}{2}}\Big\|_{L^p(\omega^r)}\lesssim C \Lambda(N)^{1-\frac{p}{2}} \Big\|\Big(\sum_{|\ell|\leq N} |g_{\ell}|^{2}\Big)^{\frac{1}{2}}\bigg\|_{L^p(\omega^r)}.
\end{align}
for some $s<r<1$. Fix $f,$ applying \eqref{vv-interpol-3} with $g_\ell=\beta_{k-l}(D)f$ we further derive 
\begin{align}
\nonumber&\Big\|\sup_{\ell\in \mathbb{Z}: |\ell|\leq N} |A_{2^\ell} \beta_{k-l}(D) f|\Big\|_{L^p(\omega^r)}\\
\nonumber&\leq \Big\|\Big(\sum_{\ell\in \mathbb{Z}: |\ell|\leq N} |A_{2^\ell}\beta_{k-l}(D)f|^2\Big)^{1/2}\Big\|_{L^p(\omega^r)}\\
&\lesssim \Lambda(N)^{1-\frac{p}{2}} \Big\|\Big(\sum_{\ell\in \mathbb{Z}: |\ell|\leq N} |\widetilde{\beta}_{k-l}(D)f|^2\Big)^{1/2}\Big\|_{L^p(\omega^r)}\lesssim \Lambda(N)^{1-\frac{p}{2}} \|f\|_{L^p(\omega^r)},\label{vv-interpol-4}
\end{align}
where the final inequality follows from the boundedness of Littlewood-Paley square-function on $L^p(\omega^r)$ for $\omega\in \mathcal{A}_{p}$ and $1<p<\infty.$ A further interpolation of \eqref{vv-interpol-4} and the unweighted estimate \eqref{unweighted-decay} for the same $1 < p < 2,$ implies that there exists $\alpha=\alpha_{s}\in (0, 1)$, $\delta^{\prime}(p) >0$ such that
\begin{equation}
\label{vv-interpol-5}
\Big\|\sup_{\ell\in \mathbb{Z}: |\ell|\leq N} |A_{2^\ell} \beta_{k-l}(D) f|\Big\|_{L^p(\omega^s)}\leq 2^{-k\delta'(p)} \Lambda(N)^{1-{\alpha}} \|f\|_{L^p(\omega^s)}.   
\end{equation}
Summing \eqref{vv-interpol-5} over $k,$ we get
\begin{equation}
\label{vv-interpol-6}
\Big\|\sup_{\ell\in \mathbb{Z}: |\ell|\leq N} |A_{2^\ell} f|\Big\|_{L^p(\omega^s)}\lesssim  \Lambda(N)^{1-{\alpha}} \|f\|_{L^p(\omega^s)}.   
\end{equation}
According to our hypothesis \eqref{Christ-est}, $\Lambda(N)$ is the smallest constant such that \eqref{Christ-est} holds, therefore, \eqref{vv-interpol-6} implies that $\Lambda(N)\lesssim \Lambda(N)^{1-{\alpha}},$ thus yielding $\Lambda(N)\lesssim 1$. This completes the proof in the case $1<p<2.$  

Now we will prove the theorem for $2\leq p<\infty.$ To start with, using a union bound argument, we observe that
\begin{align}
\nonumber&\Big\|\sup_{\ell\in \mathbb{Z}: |\ell|\leq N} |A_{2^\ell} \beta_{k-l}(D) f|\Big\|_{L^p(\omega)}^{p}\\  \nonumber&\leq \Big\|\Big(\sum_{\ell\in \mathbb{Z}: |\ell|\leq N} |A_{2^\ell} \beta_{k-l}(D) f|^p \Big)^{1/p}\Big\|_{L^p(\omega)}^{p}\\
\nonumber &\leq \sum_{\ell\in \mathbb{Z}: |\ell|\leq N} \|A_{2^\ell} \beta_{k-l}(D) f\|^{p}_{L^p(\omega)}\\
\nonumber &\overset{\eqref{uniform-est}}{\lesssim} \sum_{\ell\in \mathbb{Z}: |\ell|\leq N} \|\beta_{k-l}(D) f\|^{p}_{L^p(\omega)}\lesssim \int \Big(\sum_{\ell\in \mathbb{Z}: |\ell|\leq N} |\beta_{k-l}(D) f|^{2}\Big)^{p/2}\,\omega,\ \ \text{(as}\quad 2\leq p<\infty)\nonumber\\
&\lesssim \|f\|_{L^p(\omega)}^p,\label{vv-interpol-9}
\end{align}
where we used weighted boundedness of the square-function in the final inequality. Stein--Weiss interpolation between \eqref{vv-interpol-9} and the unweighted estimate \eqref{unweighted-decay} implies that 
\begin{align}
\Big\|\sup_{\ell\in \mathbb{Z}: |\ell|\leq N} |A_{2^\ell} \beta_{k-l}(D) f|\Big\|_{L^p(\omega^s)}\lesssim 2^{-k\tilde{\delta}(p)} \|f\|_{L^p(\omega^s)},    
\end{align}
for $0\leq s<1$ with $\tilde{\delta}(p)>0.$ Summing the above estimate in $k$ we complete the proof.
\end{proof}

\begin{proof}[Proof of Theorem~\ref{lacunary}]
Let $-1<\alpha<(p-1).$ From Proposition~\ref{single} we obtain that $A_{t}$ is bounded from $L^p(|x|^\alpha)$ to itself. An immediate application of Proposition~\ref{lac-uniform} yields that $M_{\text{lac}}$ maps $L^p(|x|^\alpha)$ to itself boundedly. For $\alpha=-1,$ one can modify the arguments in \cite{JavierVega} to show that for $\alpha<0,$ the estimate  
$$\int |M_{\text{lac}}f(x)|^p\, |x|^{\alpha}\, dx\leq C \int |f(x)|^p\, |x|^{\alpha}\, dx$$
holds true whenever 
\begin{align}
\label{unoform-end}
\int |A_{2^\ell} f(x)|^p\, |x|^{\alpha}\, dx\leq C \int |f(x)|^p\,|x|^{\alpha})\, dx     
\end{align}
holds uniformly for all $\ell \in \Z,$ and then using Proposition~\ref{single}  we conclude the end-point $\alpha=-1.$
\end{proof}

\end{document}